\documentclass{amsart}

\usepackage[mathscr]{eucal}
\usepackage{amsfonts}
\usepackage{amsmath}
\usepackage{amsthm}
\usepackage{amssymb}
\usepackage{latexsym}
\usepackage{pgf}
\usepackage{tikz}
\usepackage{tikz-cd}

\usepackage[noadjust]{cite}

\usepackage{graphicx}
\usepackage{color}
\usepackage{epsfig}

\newtheorem{theorem}{Theorem}[section]

\newtheorem{proposition}[theorem]{Proposition}
\newtheorem{corollary}[theorem]{Corollary}
\newtheorem{lemma}[theorem]{Lemma}
\newtheorem{conjecture}[theorem]{Conjecture}
\newtheorem*{conjecture*}{Conjecture}
\newtheorem*{question*}{Question}

\theoremstyle{definition}
\newtheorem{definition}[theorem]{Definition}
\newtheorem{remark}[theorem]{Remark}

\newtheorem{example}[theorem]{Example}

\let\nc\newcommand

\nc{\la}{\label}

\def\bthm{\begin{theorem}}
\def\ethm{\end{theorem}}
\def\blemma{\begin{lemma}}
\def\elemma{\end{lemma}}
\def\bproof{\begin{proof}}
\def\eproof{\end{proof}}
\def\bprop{\begin{proposition}}
\def\eprop{\end{proposition}}

\def\Z{\mathbb{Z}}

\def\N{\mathbb{N}}

\def\U{\mathcal{U}}

\def\H{\mathscr{H}}

\def\e{\boldsymbol{\mathrm{e}}}

\def\ult{{\underline{t}}}

\def\sl{\mathfrak{sl}}

\def\m{\mathfrak{m}}
\def\c{\mathbb{C}}
\def\C{\mathbb{C}}

\def\Q{\mathbb{Q}}

\nc{\Hom}{{\rm{Hom}}}
\nc{\Ext}{{\rm{Ext}}}
\nc{\htau}{{\bar{ t}}}
\nc{\HOM}{\underline{\rm{Hom}}}
\nc{\EXT}{\underline{\rm{Ext}}}
\nc{\TOR}{\underline{\rm{Tor}}}
\nc{\End}{{\rm{End}}}
\nc{\Map}{{\rm{Map}}}
\nc{\Out}{{\rm{Out}}}
\nc{\GL}{{\rm{GL}}}
\nc{\SL}{{\rm{SL}}}
\nc{\PGL}{{\rm{PGL}}}
\nc{\G}{{\rm{G}}}
\nc{\Rep}{{\rm{Rep}}}
\nc{\ad}{{\rm{ad}}}
\nc{\dlim}{\varinjlim}

\def\S{\mathrm{S}}

\newcommand{\Tr}{{\rm{Tr}}}
\newcommand{\tr}{{\rm{tr}}}

\newcommand{\ev}{{\rm{ev}}}

\numberwithin{equation}{section}

\newcommand{\PS}[1]{ {\color{red} \it \bf COMMENT****: #1 } }

\nc{\rrr}{\mathcal{R}}
\nc{\hhbar}{h}

\usepackage[final]{showkeys}

\title{Cyclotomic expansion of generalized Jones polynomials}

\author{Yuri Berest}
\address{Department of Mathematics, Cornell University, Ithaca} 
\email{berest@math.cornell.edu}

\author{Joseph Gallagher}
\address{Department of Mathematics, Cornell University, Ithaca} 
\email{jtg226@cornell.edu}

\author{Peter Samuelson}
\address{Department of Mathematics, University of California, Riverside} 
\email{psamuels@ucr.edu}

\begin{document}

\begin{abstract}
 In our previous work, \cite{BS16}, we proposed a conjecture
 that the Kauffman bracket skein module of any knot in $S^3$ 
 carries a natural action of a rank 1 double affine Hecke algebra
 $SH_{q,t_1, t_2}$ depending on 3 parameters $q, t_1, t_2$. As a consequence,
 for a knot $K$ satisfying this conjecture, we defined a three-variable 
 polynomial invariant $J^K_n(q,t_1,t_2)$ generalizing the classical colored
 Jones polynomials $J^K_n(q)$. In this paper, we give explicit formulas
 and provide a quantum group interpretation for the 
 polynomials $J^K_n(q,t_1,t_2)$. Our formulas 
 generalize
 the so-called cyclotomic expansion of the classical
 Jones polynomials constructed by K.\ Habiro \cite{Hab08}: as in the 
 classical case, they imply the integrality of $J^K_n(q,t_1,t_2)$ and,
 in fact, make sense for an arbitrary knot $K$ independent of whether or not  it satisfies the conjecture of \cite{BS16}.
 When one of the Hecke deformation parameters is set to be 1, we show that the 
 coefficients of the (generalized) cyclotomic expansion of $J^K_n(q,t_1)$
 are expressed in terms of Macdonald orthogonal polynomials.
\end{abstract}

\maketitle

\section{Introduction and statement of results}

One of the most interesting `quantum' invariants of an oriented 
$3$-manifold $ M $ studied extensively in recent years is the 
Kauffman bracket skein module $ K_q(M) $. This invariant -- introduced by 
J. Przytycki \cite{Prz91} and V. Turaev \cite{Tur91} in the early 90s -- is defined topologically 
as the quotient vector space spanned by all (framed unoriented) links in $M$
modulo the Kauffman skein relations depending on a parameter 
$q$. In \cite{BS16}, the first and third authors conjectured that the skein 
module  $ K_q(M_K) $ of the complement $ M_K := S^3 \setminus\! K $ of 
a knot in $ S^3 $ carries a \,{\it natural}\, action of a rank one 
(spherical) double affine Hecke algebra $ \mathcal{SH}_{q, t_1, t_2} $, 
which depends -- in addition to the `quantum' parameter $q$ -- on two 
new `Hecke' parameters $t_1$ and $t_2$ (see Conjecture \ref{BSConj} below).  
Our conjecture boils down 
to the assumption that $ K_q(M_K) $ possesses a certain symmetry of
algebraic nature that allows one to deform the topological action of
the skein algebra $ K_q(\partial M_K) $ of the boundary 2-torus into
the   action of $ \mathcal{SH}_{q, t_1, t_2} $. We verified our conjecture
in a number of nontrivial cases, including torus knots and 
some (non-algebraic) 2-bridge knots (see \cite{BS16, BS18}). An important consequence of this conjecture is the 
existence of polynomial knot invariants $ J^K_n(q, t_1, t_2) \in 
\C[q^{\pm 1}, t_1^{\pm 1}, t_2^{\pm 1}] $ depending on the three variables
$ q, t_1, t_2 $, which specialize (when $ t_1 = t_2 = 1$) to the classical
($\mathfrak{sl}_2$, colored) Jones polynomials $ J^K_n(q) $. We call $ J^K_n(q, t_1, t_2) $
the \emph{generalized Jones polynomials} associated to $K$.

The goal of this paper is to give an explicit formula for the
polynomials $ J^K_n(q, t_1, t_2) $ generalizing the so-called
{\it cyclotomic expansion} of the colored Jones polynomials $ J^K_n(q) $
discovered by K. Habiro. We recall that Habiro  proved in \cite{Hab08}  the
following remarkable theorem.

\begin{theorem}[\cite{Hab08}]\label{thm:hab}
For any knot $ K $ in $ S^3 $, the $n$-th colored
Jones polynomial of $K$ can be written in the form
\begin{equation}\label{eq:hab}
J_n^K(q) = \sum_{i=1}^{n} c_{n,i-1}(q)H_{i-1}^K(q)
\end{equation}
where $ H_{i-1}^K(q) \in \Z[q^{\pm 1}]$ are integral Laurent polynomials depending
on the knot $ K $ (but not on the `color' $n$), and the coefficients $c_{n,i-1}(q)$ are independent of $K$ and given by the elementary formulas 
\begin{equation}\label{eq:cyc}
c_{n,i-1} := \frac{1}{q^2-q^{-2}} \prod_{p=n-i+1}^{n+i-1} (q^{2p}-q^{-2p}),\quad \quad  1 \leq i \leq n
\end{equation}
\end{theorem}

Following \cite{GL11}, we refer to
$ H_{i-1}^K(q) $, $ i \ge 1 $, as the {\it Habiro polynomials} of $K$, while the coefficients  \eqref{eq:cyc} are called the \emph{cyclotomic coefficients}.
It is easy to show that
the $H_{i-1}^K(q)$'s always exist as rational functions in $\mathbb Q(q)$; the nontrivial part of Theorem \ref{thm:hab} is that these rational functions are actually in $\mathbb Z[q^{\pm 1}]$. 

Now, the main result of the present paper can be encapsulated in the following
\begin{theorem}
\label{thm:main}
Assume Conjecture \ref{BSConj} holds for a knot $K \subset S^3$. Then the generalized Jones polynomials $J^K_{n}(q,t_1,t_2)$ can be written in the form 
\begin{equation}\label{eq:gencyc}
J_n^K(q,t_1,t_2) = \sum_{i=1}^{n} \tilde c_{n,i-1}(q,t_1,t_2) H^K_{i-1}(q)
\end{equation}
where  $H^K_{i-1}(q)$ are the Habiro polynomials of $K$. The coefficients $\tilde{c}_{n,i-1}(q,t_1,t_2)$ are independent of $K$ and determined by the following generating function:
\begin{equation}\label{eq:genfun}
\sum_{n=0}^\infty \tilde c_{n,i-1}(q,t_1,t_2)\lambda^n = 
\frac{\det \left(B_{2i}(q,t_1,t_2;\lambda)\right)}{\prod_{N=1}^{2k-1} \gamma_N},\quad \quad i \geq 1
\end{equation}
where $B_{2i}(q,t_1,t_2;\lambda)$ is the $(2i\times 2i)$ matrix 
\begin{equation}\label{eq:bdef}
\left( 
\begin{array}{ccccccc}
0 & \alpha_1^{(i)} & 0 & \alpha_2^{(i)} & \cdots & 0 & \alpha_{i}^{(i)}\\
\beta_1 & \gamma_1 & 0 & 0 & \cdots & 0& 0\\
\beta_2 & b_{21} & \gamma_2 & 0 & \cdots & 0 & 0\\
\beta_{3} & b_{31} & b_{32} & \gamma_3 & \cdots & 0 & 0\\
\vdots & \vdots & \vdots & \vdots & \ddots & \vdots & \vdots\\
\beta_{2i-1} & b_{2i-1,1} & b_{2i-1,2} & b_{2i-1,3} & \cdots & b_{2i-1,2i-2} & \gamma_{2i-1}
\end{array}
\right)
\end{equation}
with entries (see notation in Section \ref{preliminaries})
\begin{align}
\alpha_{k}^{(i)} &:= (-1)^{i-k} \left[ \begin{array}{c}2i-1 \\ i-k\end{array}\right]_{q^2}\notag \\
b_{p,N} &:= (-1)^i \left(\{p+N\} - \{p-N\}\right)
(t_i-t_i^{-1}),
\quad\quad  i \equiv p-N+1\,\,(\mathrm{mod }\,\, 2) \label{eq:consts} \\
\beta_{N} &:= [N]_{q^2}, \notag  \quad \quad
\gamma_N := q^{2N}t_1^{-1}+q^{-2N}t_1 - \lambda - \lambda^{-1}\notag, \quad \quad 1 \leq N \leq 2i-1
\end{align}
\end{theorem}



One important consequence of formula \eqref{eq:genfun} is that the generalized cyclotomic coefficients are integral, i.e. $\tilde c_{n,i-1} \in \mathbb{Z}[q^{\pm 1}, t_1^{\pm 1}, t_2^{\pm 1}]$. In combination with Habiro's Theorem, this implies
\begin{corollary}\label{cor:integral}
The generalized Jones polynomials are integral: for all $n \geq 0$
\[
J^K_{n}(q,t_1,t_2) \in \mathbb Z[q^{\pm 1},t_1^{\pm 1},t_2^{\pm 1}]
\]
\end{corollary}
%

It is important to note that formula \eqref{eq:gencyc} of Theorem \ref{thm:main} (and Corollary \ref{cor:integral}) make sense for an \emph{arbitrary} knot $K \subset S^3$, even though they were deduced under the assumption that $K$ satisfies the conjecture of \cite{BS16}. Thus, Theorem \ref{thm:main} may be viewed as a further evidence that this conjecture holds for any knot $K$.

In the special case when $t_2=1$, we can compute the (generalized) cyclotomic coefficients $\tilde{c}_{n,k-1}(q,t_1,t_2)$ in a simple closed form using the classical Macdonald orthogonal polynomials.

\begin{theorem}\label{thm:t2eq1}
For $(t_1,t_2) =(t,1)$, the (generalized) cyclotomic coefficients in \eqref{eq:gencyc} are given by 
\begin{equation}\label{eq:yuri5}
\widetilde c_{n,i-1}(q,t) = \frac{p_{n-i}(q^{2i}t^{-1};q^{4i} | q^{4})}{p_{n-i}(q^{2i};q^{4i} | q^{4})}  \left( \prod_{k=2}^{i} \frac{q^{2k-1}t^{-1} - q^{-2k+1}t }{q^{2k-1}-q^{-2k+1}}\right) c_{n,i-1}(q)
\end{equation}
where $p_n(z;\beta | q )$ are the Macdonald symmetric polynomials of type $A_1$ 
and $c_{n,i-1}(q)$ are the classical cyclotomic coefficients  \eqref{eq:cyc}.
\end{theorem}

We remark that the Macdonald polynomials $p_n(z;\beta | q)$ can be expanded  in terms of $q$-binomial coefficients, so formulas \eqref{eq:yuri5} are entirely explicit (see Remark \ref{rmk:labelforyuri}).
The Habiro polynomials are known for certain families of knots (see, e.g. \cite{Hab08} and \cite{Mas03}). In those cases, Theorem \ref{thm:t2eq1} gives a closed form expression for generalized Jones polynomials.

\begin{example}
\noindent (1)$\,\, $ For the unknot, $H^K_{0} = 1$ and $H^K_{n}=0$ for $n \geq 1$. In this case, 
\[
J^K_{n}(q,t) = \widetilde c_{n,0}(q,t) = \frac {p_{n-1}(q^2t^{-1};q^4 | q^4)}{p_{n-1}(q^2;q^4 | q^4)}\frac{q^{2n}-q^{-2n}}{q^2-q^{-2}} = \frac{(q^2t^{-1})^n - (q^2t^{-1})^{-n}}{q^2t^{-1}-q^{-2}t}
\]
where we have used a well-known evaluation formula for Macdonald polynomials $p_{n-1}(z;q^4 | q^4) = (z^n-z^{-n})/(q^2-q^{-2})$ (see \cite[pg.\ 202]{Che05}). This recovers the result of \cite[Thm. 6.10]{BS16}.\\
\noindent (2) $\,\,$ For the figure eight knot, $H^K_{n} = 1$ for all $n \geq 0$. Hence, by Theorem \ref{thm:t2eq1},
\[
J^K_n(q,t) = \sum_{i=1}^n \frac{p_{n-i}(q^{2i}t^{-1};q^{4i} | q^{4})}{p_{n-i}(q^{2i};q^{4i} | q^{4})}  \left( \prod_{k=2}^{i} \frac{q^{2k-1}t^{-1} - q^{-2k+1}t }{q^{2k-1}-q^{-2k+1}}\right)  c_{n,i-1}
\]
Note that when $t=1$, this formula specializes to the well-known formula for the Jones polynomials of the figure 8 knot, 
\[
J_n(q) = \sum_{i=1}^n c_{n,i-1} = \frac 1 {q^2-q^{-2}} \sum_{i=1}^n \prod_{p=n-i+1}^{n+i-1} (q^{2p}-q^{-2p})
\]
\end{example}

The last result that we want to state in the Introduction provides an interpretation of our generalized Jones polynomials $J_n^K(q,t_1,t_2)$ in terms of quantum groups: more precisely, we express $J_n^K(q,t_1,t_2)$ via the universal ${\sl}_2$ invariant $J^K$ of the knot $K$ introduced by R. Lawrence  \cite{Law88, Law90} (see also \cite{Hab06, Hab08}). Recall that $J^K$ takes values in the center $\mathcal Z(\U_\hhbar)$ of the ($h$-adically) complete quantized enveloping algebra $\mathcal U_\hhbar(\mathfrak{sl}_2)$ defined over the formal power series ring $\Q_\hhbar = \Q[[\hhbar]]$ (see Section \ref{sec:uni}). We set $q = e^{\hhbar/4}$ and let $\rrr_{q,t_1,t_2} := K_0({\Rep}\, \mathcal U_\hhbar)\otimes_\Z \Q(q)[t_1^{\pm 1},t_2^{\pm 1}]$ denote the representation ring of the category ${\Rep}(\U_\hhbar)$ of finite dimensional $\U_\hhbar$-modules over the commutative ring $\Q(q)[t_1^{\pm 1},t_2^{\pm 1}]$. The ring $\rrr_{q,t_1,t_2}$ is a free module over $\Q(q)[t_q^{\pm 1},t_2^{\pm 1}]$ generated by the classes $\{[V_n]\}_{n \geq 1}$ of irreducible representations of $\U_\hhbar$; it comes together with a natural bilinear map
\begin{equation}
\tr_q(-,-): \U_\hhbar \times \rrr_{q,t_1,t_2} \to \Q_\hhbar[t_1^{\pm 1},t_2^{\pm 1}]
\end{equation}
defined by quantum traces of elements of $\U_\hhbar$ acting on finite dimensional modules (see Section \ref{sec:uni}). If $z \in \mathcal Z(\U_\hhbar)$ is a central element of $\U_\hhbar$, we write $\hat z := \tr_q(z,-): \rrr_{q,t_1,t_2} \to \Q_\hhbar[t_1^{\pm 1},t_2^{\pm 1}]$ and note that, by the Schur Lemma, 
\begin{equation}
\hat z([V_n]) = [n]_{q^2}\, z_n
\end{equation}
where $z_n$ is the scalar in $\Q_\hhbar$ by which $z$ acts on the irreducible representation $V_n$. 

Now, to state our theorem we define a sequence of functions $a_{n,p} \in \Q(q)[t_1^{\pm 1},t_2^{\pm 1}]$ (indexed by the integers $n \geq 1$ and $p\geq 0$) inductively, using the recurrence relation:
\begin{equation}\label{ynn32}
a_{n+1,p} =  A_p a_{n,p-1} + (A_p-A_{p+1})a_{n,p} + A_{-p}a_{n,p+1} - a_{n-1,p}
\end{equation}
with ``boundary''  conditions
\begin{equation}\label{ynn33}
a_{1,1}=1,\quad a_{n,0}=0, \quad a_{n,p} = 0\,\,(n \geq p),
\end{equation}
where 
\[
A_p := \frac{q^{2p-1}t_1^{-1}-q^{1-2p}t_1 + t_2 - t_2^{-1}}{q^{2p-1}-q^{1-2p}} 
\]

\begin{theorem}\label{thm:uni}
For $n \geq 1$, let $[\tilde V_n]$ denote the class in $\rrr_{q,t_1,t_2}$ given by the formula
\begin{equation}\label{ynn1}
[\tilde V_n] := \sum_{p=1}^n (-1)^{n+p}a_{n,p}\,[V_p]
\end{equation}
where the coefficients $a_{n,p} = a_{n,p}(q,t_1,t_2)$ are defined by \eqref{ynn32} and \eqref{ynn33}. Then 
\begin{equation}\label{ynn2}
J_n^K(q,t_1,t_2) = \hat J^K[\tilde V_n]
\end{equation}
\end{theorem}

Note that when $t_1=t_2=1$, we have $A_p=1$ for all $p$, and it follows easily from \eqref{ynn32} and \eqref{ynn33} that $a_{n,p}$ is equal to 1 for $p=n$ and is $0$ otherwise. Formula \eqref{ynn2} thus reduces to $J_n^K(q) = \hat J^K[V_n]$, which is a well-known formula for the colored Jones polynomials.
For arbitrary $t_1,t_2 \in \C^\ast$, one can easily compute from \eqref{ynn32} the first ``top'' terms of the sequence $\{a_{n,p}\}$:
\begin{align*}
a_{n,n} &= A_2 A_3\cdots A_n,\quad \quad (n \geq 2)\\
a_{n,n-1} &= A_2 A_3\cdots A_{n-1} (A_1 - A_n)
\end{align*}
By \eqref{ynn1}, this gives
\begin{equation}
[\tilde V_1] = [V_1],\quad \quad 
[\tilde V_2] = A_2 [V_2] + (A_2 - A_1) [V_1]
\end{equation}
In general, for $n \geq 3$, the recursive formulas for $a_{n,p}$ are more complicated: in fact, we could not find closed form expressions for these coefficients (which seems like an interesting problem).
The origin of the recurrence equations \eqref{ynn32} and \eqref{ynn33} and their relation to the double affine Hecke algebra $\H_{q,t_1,t_2}$ is explained in the proof of Lemma \ref{ynlemma1}.

The paper is organized as follows. In Section \ref{preliminaries}, we introduce notation and review basic results of \cite{BS16}, including the main conjecture of \cite{BS16} (see Section \ref{sec:conj}) and the definition of the generalized Jones polynomials $J_n^K(q,t_1,t_2)$ (see Section \ref{sec:polys}). Section \ref{sec:proofs} contains the proofs of the 3 theorems stated in the Introduction; it also fills in some details and provides definitions needed for the precise statements of these theorems. In the end of this section we mention some questions and conjectures that motivated our work.

\noindent \textbf{Acknowledgements:} We would like to thank I. Cherednik, P. Di Francesco,
N. Reshetikhin and V. Turaev for interesting
discussions, questions and comments. The work of the first author (Yu.
B.) was partially supported by
the NSF grant DMS 1702372 and the 2019 Simons Fellowship both of which
are gratefully acknowledged. The work of the third author was partially supported by a Simons Travel Grant and a Simons Collaboration Grant which are also gratefully acknowledged.


\section{Preliminaries}\label{preliminaries}
In this section we provide some background material needed for the present paper. This includes basic properties of  Kauffman bracket skein modules and double affine Hecke algebras, as well as a summary of main results of \cite{BS16}. Throughout we use the following standard notation: 
\begin{equation*}
\{n\}_q := q^n - q^{-n},  \quad \quad  [n]_q := \frac{\{n\}_q}{\{1\}_q},\quad \quad 
\left[\begin{array}{c}n\\m\end{array}\right]_q := \prod_{k=1}^m \frac{[n-k+1]_q}{[m-k+1]_q},\quad (n,m \in \N)
\end{equation*}

\subsection{Kauffman bracket skein modules}\label{kbsmsection}
A \emph{framed link} in an oriented 3-manifold $M$ is an 
embedding of a disjoint union of annuli $S^1 \times [0,1]$ into $M$, considered up to ambient isotopy. In what follows, the letter $q$ will denote either a nonzero complex number or a formal parameter generating the field 
\[
\C_q := \C(q)
\]
(we will specify which when it matters).

Let  $\mathscr L(M)$ be the vector space over $\C_q$
spanned by the set of ambient isotopy classes of framed unoriented links in $M$ (including the empty link $\varnothing$). Let $\mathscr L'(M)$ denote the smallest subspace of $\mathscr L(M)$ containing the skein expressions  

\tikzset{
        ribbon/.style={
            preaction={
                preaction={
                    draw,
                    line width=0.25cm,
                    white
                },
                draw,
                line width=0.2cm,
                black!30!#1
            },
            line width=0.15cm,
            #1
        },
        ribbon/.default=gray
    }
\begin{figure}[h]
\centering
\begin{tikzpicture}
\draw[ribbon] (-.4,0) to[out=90,in=180] (0,.4) to[out=0,in=90] (.4,0) to[out=-90, in=0] (0,-.4) to[out=180,in=-90] (-.4,0);
\draw[thick, style={dashed}] (0,0) circle(.75);
\draw (1.1,0) node[fill=white] {$+$};
\draw (2.4,0) node[fill=white] {$(q^2+q^{-2})$};
\draw[thick, style={dashed}] (4.3,0) circle(.75);
\end{tikzpicture}
\label{FrNorm}
\end{figure}

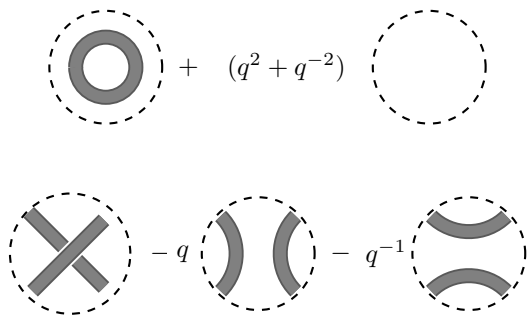
\begin{figure}[h]
\centering
\begin{tikzpicture}
\draw[ribbon] (-.05,1.05) to[out=-45, in=-45] (.95,.05);
\draw[ribbon] (0,0) to[out=45, in=45] (.95,.95);
\draw[thick, style={dashed}] (.5,.5) circle(.8);
\draw (1.7, .5) node[fill=white] {$-$};
\draw (2, .5) node[fill=white] {$q$};
\draw[ribbon] (2.5,0) to[out=45, in=-45] (2.5,1);
\draw[ribbon] (3.5,0) to[out=135, in=225] (3.5,1);
\draw[thick, style={dashed}] (3,.5) circle(.75);
\draw (4.1, .5) node[fill=white] {$-$};
\draw (4.7, .5) node[fill=white] {$q^{-1}$};
\draw[ribbon] (5.3,1) to[out=-45,in=225] (6.3,1);
\draw[ribbon] (5.3,0) to[out=45,in=135] (6.3,0);
\draw[thick, style={dashed}] (5.8,.5) circle(.75);   
\end{tikzpicture}
\label{FrSkein}

\caption{Framed skein relations}
\end{figure}
\noindent where the diagrams represent embeddings of annuli which are identical outside of the oriented $3$-ball represented by the dotted circle.

\begin{definition}[\cite{Prz91}]
The \emph{Kauffman bracket skein module} of an oriented $3$-manifold $M$ is the quotient vector space $K_q(M) := \mathscr L(M) / \mathscr L'(M)$. It contains a canonical element $\varnothing \in K_q(M)$ corresponding to the empty link.
\end{definition}

\begin{remark}
If F is a surface, we will often write $K_q(F)$ for the skein module $K_q(F\times [0,1])$ of the cylinder over $F$.
\end{remark}

In general, $K_q(M)$  carries only a linear structure. However, 
the assignment $M \to K_q(M)$ is functorial with respect to oriented embeddings, which implies the following facts:

\begin{enumerate}
\item If $M = M_1 \sqcup M_2$, then $K_q(M) \cong K_q(M_1) \otimes K_q(M_2)$.
\item For any surface $F$, the embedding $[1, 2/3] \sqcup [1/3, 0] \to [0,1]$ induces a map 
\[
\mu : K_q(F) \otimes K_q(F) \to K_q(F)
\]
which make $K_q(F)$ an associative unital algebra (with unit $\varnothing$).
\item If $\partial M \cong F$ and if $M = F \times [0,1] \sqcup N$ represents a decomposition of $M$ into a tubular neighborhood of the boundary and a retract $N \cong M$, the map
\[
m : K_q(F) \otimes K_q(M) \to K_q(M)
\]
gives $M$ the structure of a left module over $K_q(F)$. 
\end{enumerate}

\begin{example}\label{skeins3}
An original motivation for defining $K_q(M)$ was a theorem of Kauffman \cite{Kau87} asserting that the natural map
\[
\C_q \stackrel \sim \to K_q(S^3), \quad\quad  1 \mapsto \varnothing  \]
is an isomorphism of vector spaces, 
and that the inverse image of a link $L$ in $S^3$ under this isomorphism is the Jones polynomial of $L$. Clearly $K_q(S^3)$ is of dimension at most $1$ over $\C_q$ thanks to the skein relations; the key point of Kauffman's theorem is that this map is injective.
\end{example}

\begin{example}
 Let $M = S^1 \times D^2$ be the solid torus, or complement of the unknot. If $x$ is the nontrivial loop, then the map $\C_q[x] \to K_q(S^1 \times D^2)$ sending $x^n$ to $n$ parallel copies of $x$ is surjective (because all crossings and trivial loops can be removed using the skein relations). Less obvious is the fact that this map is injective and thus an isomorphism (see, e.g., \cite{SW07}).
\end{example}

\subsubsection{The Kauffman bracket skein module of the torus}
Recall that the \emph{quantum Weyl algebra} (or  \emph{quantum torus}) is defined by 
\[
A_q := \frac{\C[q^{\pm 1}] \langle X^{\pm 1},Y^{\pm 1}\rangle}{(XY-q^2YX)}
\]
Note that this algebra carries a $\Z_2$ action defined by the automorphism $(X,Y) \mapsto (X^{-1},Y^{-1})$.

We now recall a theorem of Frohman and Gelca \cite{FG00} that gives a connection between $K_q(T^2)$ and the invariant subalgebra $A_q^{\Z_2}$. 
Let $T_n \in \c[x]$ be the Chebyshev polynomials defined by
\[
T_0 = 2, \quad T_1 = x, \quad T_{n+1} = xT_n-T_{n-1}.
\]
If $m,l$ are relatively prime, write $(m,l)$ for the $m,l$ curve on the torus (the simple curve wrapping around 
the torus $l$ times in the longitudinal direction and $m$ times in the meridian's direction). It is clear that the links $(m,l)^n$ span $K_q(T^2)$, and it follows from \cite{SW07} that this set is actually a basis.  However, a more convenient basis is given by the elements $(m,l)_T := T_d((\frac m {d}, \frac l {d}))$ (where $d = \mathrm{gcd}(m,l)$). 
Define $e_{r,s} = q^{-rs}X^{r}Y^s \in A_q$, which form a linear basis in $A_q$.

\begin{theorem}[\cite{FG00}]\label{fg00}
The map $K_q(T^2) \to A_q^{\Z_2}$ given by $(m,l)_T\mapsto e_{m,l}+e_{-m,-l}$ is an isomorphism of algebras.
\end{theorem}

\begin{remark}\label{remark_canonicalmodulestructure}
 If $K$ is an oriented knot, then the meridian/longitude pair $(m,l)$ gives a canonical identification of $S^1\times S^1$ with the boundary of $S^3\setminus K$. If the orientation of $K$ is reversed, this identification is twisted by the `hyper-elliptic involution' of $S^1\times S^1$ (which negates both components). However, this induces the identity isomorphism on $K_q(T^2\times [0,1])$, so the $A_q^{\Z_2}$-module structure on $K_q(S^3\setminus K)$ is canonical and does not depend on the choice of orientation of $K$.
\end{remark}

\subsubsection{Topological pairings and colored Jones polynomials}\label{topologicalpairing}
Let $M$ be any closed $3$-manifold. If $(M_1, M_2)$ represents a Heegaard splitting of $M$, that is, $M_1, M_2 \subset M$ are oriented submanifolds with boundary satisfying
\[
M_1 \cup M_2 = M \quad \quad M_1 \cap M_2 = \partial M_1 = \partial M_2 = F,
\] 
the inclusion $\iota : M_1 \sqcup M_2 \to M$ determines by functoriality a map
\begin{equation}
\label{pairRaw}
K_q(\iota) : K_q(M_1) \otimes K_q(M_2) \to K_q(M)
\end{equation}
Now put an orientation on $F$ as the boundary of $M_2$, and let $N_F \subset M$ be a tubular neighborhood of $F$ with respect to this orientation. Let $\iota_i : N_F \to M_i, i \in \{ 1, 2 \}$ be the natural inclusions. As usual, $\iota_2$ gives $M_2$ the structure of a left module over $K_q(N_F)$. However, as the orientation of $F$ is reversed from that of $\partial M_1$, the map $\iota_1$ gives $K_q(M_1)$ the structure of a \emph{right} module over $K_q(N_F)$. As a skein in $N_F$ can be pushed into either $M_1$ or $M_2$, this tells us that (\ref{pairRaw}) actually factors as a map
\[
K_q(\iota) : K_q(M_1) \otimes_{K_q(F)} K_q(M_2) \to K_q(M).
\]
If $M = S^3$, then $M_1$ is the tubular neighborhood of a knot $K$ and $M_2 = S^3 \setminus K$, and we refer to this map as the \emph{topological pairing}
\begin{equation}
\label{topPairing}
\langle - , - \rangle : K_q(S^1 \times D^2) \otimes_{K_q(T^2)} K_q(S^3 \setminus K) \to \C_q
\end{equation}

The colored Jones polynomials $J^K_{n}(q) \in \C[q^{\pm 1}]$ of a knot $K \subset S^3$ were originally defined by Reshetikhin and Turaev in \cite{RT90} using the representation theory of $\U_q(\sl_2)$. Here we recall a theorem of Kirby and Melvin that shows how $J^K_{n}(q)$ can be computed in terms of the topological pairing.

If $D^2 \times S^1$ is a tubular neighborhood of the knot $K$, then we identify $K_q(D^2 \times S^1)\cong \C_q[u]$, where $u \in K_q(D^2 \times S^1)$ is the image of the (0-framed) longitude $l \in K_q(\partial (S^3 \setminus K))$. 
Let $S_n \in \C_q[u]$ be the Chebyshev polynomials of the second kind, which satisfy the initial conditions $S_0 = 1$ and $S_1 = u$, and the recursion relation $S_{n+1} = uS_n - S_{n-1}$.

\begin{theorem}[\cite{KM91}]\label{thm_coloredjonespolys}
 If $\varnothing \in K_q(S^3\setminus K)$ is the empty link, we have
 \[
  J^K_{n}(q) = (-1)^{n-1}\langle S_{n-1}(u), \varnothing \rangle
 \]
\end{theorem}

As the zero-framed longitude $l$ considered as an element in the skein module of the boundary torus $K_q(T^2)$ is identified with $Y + Y^{-1}$ under Theorem \ref{fg00}, we have
\begin{align}
\label{KirbyMel}
J^K_n(q) &= (-1)^{n-1}\langle \varnothing \cdot S_{n-1}(Y + Y^{-1}) , \varnothing \rangle \\
&= (-1)^{n-1}\langle \varnothing , S_{n-1}(Y + Y^{-1}) \cdot \varnothing \rangle \notag
\end{align}

\begin{remark}\label{remark_signconvention}
The sign correction is chosen so that for the unknot we have $J_{n}(q) = [n]_{q^2} = (q^{2n}-q^{-2n})/(q^2-q^{-2})$. Also, with this normalization, $J^K_{0}(q) = 0$ and $J^K_{1}(q) = 1$ for every knot $K$. This agrees  with the convention of labelling irreducible representations of $\U_q(\sl_2)$ by their dimension.
\end{remark}


\subsection{The  double affine Hecke algebra}
In this section we define a 5-parameter family of algebras $\H_{q,\ult}$ -- called the double affine Hecke algebra of type $C^\vee C_1$ -- originally introduced in  \cite{Sah99} (see also \cite{NS04} and  \cite{BS16} for our present notation).
This family represents the universal deformation of the algebra $\C[X^{\pm 1}, Y^{\pm 1}] \rtimes \Z_2$, the crossed product of the Laurent polynomial ring $\C[X^{\pm 1}, Y^{\pm 1}]$, with $\Z_2$ acting by the natural involution
(see \cite{Obl04}). The algebra $\H_{q,\ult}$ for $q \in \C^\ast$ and $\ult = (t_1,t_2,t_3,t_4) \in (\C^\ast)^4$ is generated by the elements $T_1$, $T_2$, $T_3$, and $T_4$ subject to the five  relations
\begin{align}\label{ccdaharelations}
 (T_1-t_1)(T_1+t_1^{-1}) &= 0\notag\\
 (T_2-t_2)(T_2+t_2^{-1}) &= 0\notag\\
 (T_3-t_3)(T_3+t_3^{-1}) &= 0\\
 (T_4-t_4)(T_4+t_4^{-1}) &= 0\notag\\
 T_4T_3T_1T_2 &= q\notag
\end{align}

Recall that, by definition, the crossed product algebra $A_q \rtimes \mathbb{Z}_2$ is generated by $X, Y, s$, satisfying 
\[
sX = X^{-1}s, \quad sY = Y^{-1}s, \quad s^2 = 1, \quad XY = q^2YX.
\]
Let $D_q:=\C_q(X)[Y^{\pm 1}] / (XY-q^2YX)$ denote the \emph{localized} quantum Weyl algebra obtained from $A_q$ by inverting all (nonzero) polynomials in $X$. Note that the action of $\Z_2$ extends to $D_q$ so that we can form the crossed product $D_q \rtimes \Z_2$. Now, consider the following elements in $D_q\rtimes \Z_2$: 
\begin{eqnarray}
 \hat T_1 &:=& t_1 sY + \frac{q \bar t_1 X  + \bar t_2 }{qX - q^{-1}X^{-1}}(1-sY) \label{eq:T1}\\
 \hat T_3 &:=& t_3s + \frac{\bar t_3 +\bar t_4X}{1-X^2}(1-s)\notag
\end{eqnarray}
These elements are called the Dunkl-Cherednik and Demazure-Lusztig operators, respectively.
The next proposition establishes the relation between the algebras $\H_{q,\ult}$ and $A_q\rtimes \Z_2$.
\begin{proposition}[{\cite{Sah99}, see also \cite[Thm. 2.22]{NS04}}]\label{prop_dunklembedding}
 The assignment
 \begin{equation}\label{ccdunklembedding}
 T_1 \mapsto \hat T_1, \quad T_3 \mapsto \hat T_3, \quad T_2 \mapsto q\hat T_1^{-1}X,\quad T_4 \mapsto X^{-1}\hat T_3^{-1}
\end{equation}
extends to an injective algebra homomorphism $\H_{q,\ult} \hookrightarrow D_q\rtimes \Z_2$.
\end{proposition}


Note that $A_q \rtimes \mathbb{Z}_2$ embeds in $D_q\rtimes \Z_2$ via the natural localization map. When $\ult = \underline{1}$, the assignment in (\ref{ccdunklembedding}) becomes
\begin{equation}\label{eq:rep}
T_1 \mapsto sY, \quad T_3 \mapsto s, \quad T_2 \mapsto qsY X, \quad T_4 \mapsto sX
\end{equation}
and the image of $\H_{q, \underline{1}}$ coincides with the image of $A_q \rtimes \mathbb{Z}_2$. Thus, using \eqref{eq:rep}, we can identify $\H_{q, \underline{1}} \cong A_q \rtimes \mathbb{Z}_2$.

\begin{remark}\label{remark_ccother}
The algebra $\H_{q,\ult}$ is also generated by the (invertible) elements 
\[
X := q^{-1}T_1T_2,\quad\quad  Y := T_3T_1, \quad \quad T := T_3
\] 
which satisfy the relations
\begin{eqnarray}\label{equation_ccxyt}
XT &=& T^{-1}X^{-1} - \bar t_4\notag\\
T^{-1}Y &=& Y^{-1}T + \bar t_1\notag\\
T^2 &=& 1 + \bar t_3T\notag\\
TXY &=& q^2T^{-1}YX - q^2 \bar t_1 X - q \bar t_2  - \bar t_4Y
\end{eqnarray}
where $\bar t_i = t_i - t_i^{-1}$. With this presentation it is immediate that $\H_{q,1} \cong A_q\rtimes \Z_2$.
Note that while the operator $X$ does not depend on $\ult$, the operator $Y$ does. We will write this last operator  as $Y_\ult$ when we want to stress its dependence on $\ult$. Explicitly, we have 
\[
Y_\ult = \hat T_3\,  \hat T_1
\] 
where $\hat T_1$ and $\hat T_3$ are given by formulas \eqref{eq:T1}.
\end{remark}

The following simple observation can be regarded as a motivation for the main conjecture of \cite{BS16}. For $f(X) \in \C(X)$, define the operators
\[
Y\cdot f(X) = f(q^{-2}X), \quad\quad  X\cdot f(X) = Xf(X), \quad\quad  s\cdot f(X) = f(X^{-1})
\]
These operators give $\C(X)$ the structure of a left $D_q\rtimes \Z_2$-module. The subspace $\C[X^{\pm 1}] \subset \C(X)$ is obviously preserved by $A_q \rtimes \mathbb{Z}_2$ and is called the \emph{polynomial representation}. A remarkable fact (which can be checked by direct calculation)  is that $\C[X^{\pm 1}]$ is also preserved by $\H_{q, \ult}$ (for all $\ult$) under the action of \eqref{ccdunklembedding}. This gives the \emph{polynomial representation} of $\H_{q,\ult}$, which can thus be viewed as a deformation of the polynomial representation of $A_q\rtimes \Z_2$.

The element $\e := (T_3+t_3^{-1})/(t_3+t_3^{-1})$ is an idempotent in $\H_{q,\ult}$, and the algebra $\S\H_{q,t} := \e \H_{q,\ult}\e$ is called the \emph{spherical subalgebra} of $\H_{q,\ult}$.  It is easy to check that $\e$ commutes with $X+X^{-1}$ and that the subspace $\e \cdot \C[X^{\pm 1}] \subset \C[X^{\pm 1}]$ is equal to the subspace $\C[X+X^{-1}]$ of symmetric polynomials in $\C[X^{\pm 1}]$. The spherical algebra therefore acts on $\C[X+X^{-1}]$, and this module is called the \emph{symmetric polynomial representation} of $\S \H_{q,\ult}$.


\subsection{Main conjecture of \cite{BS16}}\label{sec:conj}
We first recall that the algebras $A_q\rtimes \mathbb{Z}_2$ and $A_q^{\mathbb Z_2}$ are Morita equivalent. More precisely, if $q^4 - 1$ is invertible, then the functors
\begin{align}
\e A \otimes_{A} - &: \rm{Mod}(A) \to \rm{Mod}(\e A \e) \notag\\
A\e \otimes_{\e A \e} - &: \rm{Mod}(\e A \e) \to \rm{Mod}(A)
\label{eq:morita}
\end{align}
are mutually inverse equivalences of categories. 

We can identify $(A_q\rtimes \mathbb Z_2)\e = A_q$ as left $A_q\rtimes \mathbb Z_2$-modules, and $\e A \e \cong A_q^{\mathbb Z_2}$ as $\C_q$-algebras. Let $K$ be a knot in $S^3$, so that $K_q(S^3 \setminus K)$ has the canonical structure of a left $A_q^{\mathbb{Z}_2}$-module. Applying the previous proposition, we may form the \emph{nonsymmetric skein module} $\widehat{K}_q(S^3 \setminus K)$
\[
\widehat{K}_q(S^3 \setminus K) := A_q \otimes_{A_q^{\mathbb{Z}_2}} K_q(S^3 \setminus K).
\] 
This is naturally a left $A_q \rtimes \mathbb{Z}_2$-module, and so we may localize it at all nonzero polynomials in $X$. Call the resulting $D_q\rtimes \Z_2$-module $\widehat{K}^{loc}_q(S^3 \setminus K)$, i.e.
\[
\widehat{K}^{loc}_q(S^3 \setminus K) := (D_q\rtimes \Z_2) \otimes_{A_q \rtimes \mathbb{Z}_2} \widehat{K}_q(S^3 \setminus K)
\]
By Proposition \ref{ccdunklembedding}, $\widehat{K}^{loc}_q(S^3 \setminus K)$ is then a $\mathscr{H}_{q, (t_1, t_2, t_3, t_4)}$-module. 

\begin{example} Let $K$ be the unknot. 
\label{unknotStruct}
In this case, $\widehat{K}_q(S^3 \setminus K) \cong \C_q[X^{\pm 1}]$ as a $\C_q[X^{\pm 1}]$-module. The action of the generators $Y,s \in A_q \rtimes \mathbb{Z}_2$ is given by the formulas
\[
Y\cdot f(X) := -f(q^{-2}X), \quad\quad  s\cdot f(X) = -f(X^{-1})
\]
The localized skein module $\widehat{K}^{loc}_q(S^3 \setminus K)$ is simply $\C_q(X)$. Thus in this case the natural localization map
\[
\eta : \widehat{K}_q(S^3 \setminus K) \to \widehat{K}^{loc}_q(S^3 \setminus K)
\]
is injective, and we can identify $\widehat{K}_q(S^3 \setminus K)$ with its image under $\eta$. We want to know if the $\H_{q,\ult}$ action preserves this image as in the case of the polynomial representation.

Recall that by Remark \ref{remark_ccother}, the algebra $\H_{q,\ult}$ is generated by the operators $X, T_1, T_3$, which act on polynomials by formulas \eqref{eq:T1}:
\begin{align*}
T_1 \cdot X^n &= t_1q^{2n}X^{-n} + q^{-2n}X^{-n}(q^2\overline{t_1}X^2 + q\overline{t_2}X) \frac{1-q^{2n}X^{2n}}{1-q^2X^2}\\
T_3 \cdot X^n &= -t_3X^{-n} + (\overline{t_3} + \overline{t_4}X) \frac{X^n+X^{-n}}{1-X^2}
\end{align*}
\end{example}
We see that $T_1$ always preserves $\widehat{K}_q(S^3 \setminus K) \subset \widehat{K}^{loc}_q(S^3 \setminus K)$, while  $T_3$ preserves this subspace only when $t_3 = t_4 = 1$.
%
%
%
%
Conjecturally, this behavior generalizes to all knots. To be precise, we have
\begin{conjecture}[{\cite{BS16}}]
\label{BSConj}
For all knots $K \subset S^3$, the following are true:
\begin{enumerate}
\item The localization map $\eta : \widehat{K}_q(S^3 \setminus K) \to \widehat{K}_q^{loc}(S^3 \setminus K)$ is injective.
\item The natural action of $\mathscr{H}_{q, (t_1, t_2,1,1)}$ on $\widehat{K}_q^{loc}(S^3 \setminus K)$ preserves the subspace $\widehat{K}_q(S^3 \setminus K)$, the image of the localization map $\eta$.
\end{enumerate}
\end{conjecture}

By symmetrization, the second statement of Conjecture \ref{BSConj} implies that the spherical subalgebra $\S\H_{q,t_1,t_2,1,1}$ acts on the skein module $K_q(S^3\setminus K)$ itself. It is shown in \cite{BS16} and \cite{BS18} that this holds in many cases: for the unknot, figure eight, and $(2,2p+1)$-torus knots for generic $q$, and for 2-bridge knots, all torus knots, and connect sums of such when $q=-1$.

\subsection{The generalized Jones polynomials}\label{sec:polys}

An interesting consequence of Conjecture \ref{BSConj} is the existence of a multivariable generalization of the (colored) Jones polynomials $J_n^K(q)$. Recall, by Theorem \ref{thm_coloredjonespolys}, $J_n^K(q)$ can be computed using the natural (topological) pairing of the Kauffman bracket skein modules, by the Kirby-Melvin formula (cf. \eqref{KirbyMel}):
\begin{equation}\label{yn1}
J_n^K(q) = (-1)^{n-1} \langle \varnothing, S_{n-1}(Y+Y^{-1})\cdot \varnothing\rangle
\end{equation}
Under the Morita equivalence \ref{eq:morita}, the topological pairing $\langle -,-\rangle$ extends uniquely to a bilinear pairing of \emph{nonsymmetric} skein modules\footnote{Abusing notation, we denote the extended pairing of nonsymmetric skein modules in the same way as the ``symmetric'' (topological) one.}:
\begin{equation}\label{yn2}
\langle -,-\rangle: \hat K_q(S^1\times D^2) \times \hat K_q(S^3\setminus K) \to \C_q
\end{equation}
and formula \eqref{yn1} still holds for this extended pairing (see \cite[Cor. 5.3]{BS16}). We note that by construction, this bilinear pairing is in fact balanced over $A_q\rtimes \Z_2$, i.e.\ it induces a $\C_q$-linear map 
\begin{equation}\label{eq:yuriedits}
\langle -,-\rangle: \hat K_q (S^1\times D^2) \otimes_{A_q\rtimes \Z_2} \hat K_q(S^3\setminus K) \to \C_q
\end{equation}
The right action of $A_q\rtimes \Z_2$ on $\hat K_q(S^1\times D^2)$ in \eqref{eq:yuriedits} is described explicitly in \cite[Lemma 5.5]{BS16}.
Specifically, $\hat K_q(S^1\times D^2)$ can be identified with the space of Laurent polynomials $\C_q[U^{\pm 1}]$ with $A_q\rtimes \Z_2$ acting by 
\begin{align}
f(U)\cdot Y &:= f(U)\cdot U^{-1}\notag\\
f(U) \cdot X &:= -f(q^2 U) \label{yn210}\\
f(U) \cdot s &:= -f(U^{-1})\notag
\end{align}
The distinguished element (``empty link'') $\varnothing$ in $\hat K_q(S^1\times D^2)$ corresponds under this identification to the element $U-U^{-1} \in \C_q[U^{\pm 1}]$, which we still denote by $\varnothing$.

When a knot $K$ satisfies Conjecture \ref{BSConj}, the nonsymmetric skein module $\hat K_q(S^3\setminus K)$ carries a natural action of the DAHA $\H_{q,t_1,t_2}$ and the ``longitude'' operator $Y$ admits a natural deformation to the DAHA operator $Y_{t_1,t_2} := T_3 T_1$ (see Remark \ref{remark_ccother}). This motivates the following.

\begin{definition}[\cite{BS16}]\label{y213}
Assume that $K \subset S^3$ satisfies Conjecture \ref{BSConj}. Then we define the \emph{generalized Jones polynomial} of $K$ by 
\begin{equation}\label{yn3}
J_n^K(q,t_1,t_2) := (-1)^{n-1} \langle \varnothing, S_{n-1}(Y_{t_1,t_2} + Y_{t_1,t_2}^{-1}) \cdot \varnothing\rangle
\end{equation}
where $\langle -,-\rangle$ is the extended topological pairing  \eqref{yn2}.
\end{definition}

Note that formula \eqref{yn3} makes sense precisely because, by Conjecture \ref{BSConj}, the skein module $\hat K_q(S^3\setminus K)$ is a module over $\H_{q,t_1,t_2}$.
When $t_1=t_2=1$, it reduces to the Kirby-Melvin formula \eqref{yn1}, and we have $J_n(q,1,1) = J_n^K(q)$. The generalized Jones polynomial $J_n^K(q,t_1,t_2)$ can be thus viewed as a two-parameter (``Hecke'') deformation of $J_n^K(q)$. 

\section{Proofs}\label{sec:proofs}
In this section, we prove our three main theorems stated in the Introduction. 

\subsection{The deformed pairing}


To compute the generalized Jones polynomials \eqref{yn3}, we need a ``deformed'' version of formula \eqref{KirbyMel}, which leads us to the natural question: Is the topological  pairing \eqref{yn2} balanced over $\H_{q,t_1,t_2}$ for $t_1,t_2 \not= 1$? The (affirmative) answer to this question is the starting point for our calculations:

\begin{lemma}\label{yn214}
Assume that a knot $K \subset S^3$ satisfies  Conjecture \ref{BSConj}. Then, for any $t_1,t_2 \in \C^\ast$, the pairing \eqref{yn2} induces a linear map 
\begin{equation}\label{yn4}
\langle -,-\rangle: \hat K_q(S^1\times D^2) \otimes_{\H_{q,t_1,t_2}} \hat K_q(S^3\setminus K) \to \C_q(t_1,t_2)
\end{equation}
where the (right) $\H_{q,t_1,t_2}$-module structure on $\hat K_q(S^1\times D^2)$ is defined by \eqref{yn210} via the Demazure-Lusztig and Dunkl-Cherednik operators \eqref{eq:T1}.
\end{lemma}
\begin{proof}
Recall that  the pairing \eqref{yn2} is balanced over $A_q\rtimes \Z_2$ (see \eqref{eq:yuriedits}). To prove the lemma, it is sufficient to show that it is balanced over an invertible generating set of $\H_{q, (t_1, t_2)}$, which we take to be $X$, $s$, and the operator
\[
T_1 = t_1sY-\frac{q^2\bar{t}_1X^2 + q\overline{t}_2X}{1-q^2X^2}(1-sY)
\]
Since the pairing is already balanced over $s$ and $X$ and $Y$, it will suffice to show it is balanced with respect to $\frac{1}{1-q^2X^2}(1-sY)$.  Since the image of 
$\widehat{K}_q(S^3 \setminus K)$ is preserved in its localization, if $m \in \widehat{K}_q(S^3 \setminus K)$ then there exists a unique $m' \in \widehat{K}_q(S^3\setminus K)$ such that
\[
(1-sY)\cdot m = (1-q^2X^2)\cdot m'
\]
Thus we can compute
\[
\langle U^k , \frac{1}{1-q^2X^2}(1-sY) \cdot m \rangle = \langle U^k , m' \rangle.
\]
On the other hand, acting on the right by the same operator gives 
\begin{align*}
\langle U^k \cdot \frac{1}{1-q^2X^2}(1-sY), m \rangle &= \langle U^k.\frac{1}{1-q^2X^2} , (1-sY) \cdot m \rangle\\
&= \langle U^k \cdot\frac{1}{1-q^2X^2}  , (1-q^2X^2)m' \rangle\\
&= \langle U^k , m' \rangle
\end{align*}
This completes the proof of Lemma \ref{yn214}.
\end{proof}

\begin{corollary}\label{yn214p}
If $K$ satisfies Conjecture \ref{BSConj}, then 
\begin{equation}\label{yn5}
J_n^K(q,t_1,t_2) = (-1)^{n-1}\langle \varnothing \cdot S_{n-1}(Y_{t_1,t_2}+Y_{t_1,t_2}^{-1}),\varnothing\rangle
\end{equation}
\end{corollary}
\begin{proof}
Formula \eqref{yn5} is immediate from Definition \ref{y213} and Lemma \ref{yn214}.
\end{proof}

\subsection{Proof of Theorem \ref{thm:main}}
From now, we fix a knot $K \subset S^3$ and (unless otherwise stated) assume that it satisfies the conditions of Conjecture \ref{BSConj}.

\begin{lemma}\label{ynlemma1}
For all $n \geq 0$, 
\begin{equation}\label{yn6}
J_n^K(q,t_1,t_2) = \sum_{p=1}^n (-1)^{n+p} a_{n,p}(q,t_1,t_2)\, J_p^K(q),
\end{equation}
where the coefficients $a_{n,p} = a_{n,p}(q,t_1,t_2)$ are defined in the Introduction (see \eqref{ynn32} and \eqref{ynn33}).
\end{lemma}
\begin{remark}
We note that $a_{n,k}(q,t_1,t_2)$ in \eqref{yn6} are \emph{rational} functions of $q$, and it is by no means obvious that the right-hand side of formula \eqref{yn6} is polynomial in $q$. We will show later -- invoking the Habiro Theorem --  that this is indeed the case for \emph{any} knot $K$, whether or not it satisfies Conjecture \ref{BSConj}.
\end{remark}
\begin{proof}[Proof of Lemma \ref{ynlemma1}]
Recall that under the identification $\hat K_q(S^1\times D^2) \cong \C_q[U^{\pm 1}]$ (see \eqref{yn3}), the empty link $\varnothing$ in $\hat K_q(S^2\times D^2)$ corresponds to the element $U-U^{-1}\in \C_q[U^{\pm 1}]$. The operators $S_{n-1}(Y_{t_1,t_2}+Y^{-1}_{t_1,t_2})$ are invariant under (i.e. commute with) the action of $\Z_2$ on $\C_q[U^{\pm 1}]$. Hence, for all $n \geq 1$, we can expand $(U-U^{-1})\cdot S_{n-1}(Y_{t_1,t_2}+Y_{t_1,t_2}^{-1})$ in $\C_q[U^{\pm 1}]$ as 
\begin{equation}\label{yn9}
(U-U^{-1})\cdot S_{n-1}(Y_{t_1,t_2}+Y^{-1}_{t_1,t_2}) = \sum_{p=1}^n a_{n,p} (U^p-U^{-p})
\end{equation}
for some (uniquely determined) coefficients $\tilde a_{n,p} \in \C_q(t_1,t_2)$. By Corollary \ref{yn214p}, this gives
\begin{align*}
J_n(q,t_1,t_2) &= (-1)^{n-1}\langle \varnothing \cdot S_{n-1}(Y_{t_1,t_2}+Y^{-1}_{t_1,t_2}), \varnothing \rangle\\
&= (-1)^{n-1} \langle (U-U^{-1})\cdot S_{n-1}(Y_{t_1,t_2}+Y^{-1}_{t_1,t_2}),\varnothing\rangle\\
&= (-1)^{n-1}\sum_{p=1}^n  a_{n,p} \langle U^p-U^{-p}, \varnothing\rangle\\
&= \sum_{p=1}^n (-1)^{n+p}  a_{n,p} \, J_p(q)
\end{align*}
where the last equality is the consequence of the Kirby-Melvin formula \eqref{KirbyMel} (cf. \cite[Lemma 5.6]{BS16}). Thus, to complete the proof of the lemma it suffices to show that  the coefficients $ a_{n,p} $ in \eqref{yn9} are determined precisely by the relations \eqref{ynn32} and \eqref{ynn33}. This can be done by a lengthy but straightforward induction (in $n$) using the defining relations $S_n = uS_{n-1} - S_{n-2}$ for the Chebyshev polynomials. We leave this calculation as an exercise for the reader.
\end{proof}

Combining formula \eqref{yn6} of Lemma \ref{ynlemma1} with Habiro's expansion of the classical Jones polynomials (see Theorem \ref{thm:hab}), we get
\begin{align*}
J_n(q,t_1,t_2) &= \sum_{p=1}^n (-1)^{n+p} a_{n,p}\, J_p(q) = \sum_{p=1}^n (-1)^{n+p} a_{n,p}\left( \sum_{i=1}^p c_{p,i-1}\, H_{i-1}\right)\\
&= \sum_{p=1}^n \sum_{i=1}^p (-1)^{n+p} a_{n,p}\, c_{p,i-1}\,  H_{i-1} = \sum_{i=1}^n\left( \sum_{p=i}^n (-1)^{n+p}a_{n,p}\, c_{p,i-1}\right) H_{i-1}
\end{align*}
where $H_{i-1} = H_{i-1}(q) $ are the Habiro polynomials of the knot $K$ and $c_{p,i-1}$ are the classical cyclotomic coefficients defined by formula \eqref{eq:cyc}. Since $c_{p,i-1}\equiv 0$ for $p < i$, we can rewrite the last formula in the form
\begin{equation}\label{yn10}
J_n(q,t_1,t_2) = \sum_{i=1}^n \tilde c_{n,i-1} \, H_{i-1}
\end{equation}
where 
\begin{equation}\label{yn11} \tilde c_{n,i-1} := \sum_{p=1}^n (-1)^{n+p} a_{n,p}\, c_{p,i-1}
\end{equation}

Now, to prove Theorem \ref{thm:main} we need to compute the generating functions $G_i(\lambda):=\sum_{n=0}^\infty \tilde c_{n,i-1} \lambda^n$. Using \eqref{yn11} we can write these  functions in the form
\begin{equation}\label{yn12}
G_i(\lambda) = \sum_{n=0}^\infty \sum_{p=1}^n (-1)^p a_{n,p} c_{p,i-1} (-\lambda)^n,\quad \quad i \geq 1
\end{equation}
Formula \eqref{yn11} suggests that $G_i(\lambda)$ may be expressed in a simple way in terms of the generating series of the double sequence $\{a_{n,p}\}$:
\begin{equation}\label{yn13}
F(U,\lambda) := \sum_{n=0}^\infty \sum_{p=-\infty}^\infty a_{n,p}U^p\lambda^n
\end{equation}
which we define by formally extending the functions $p \mapsto a_{n,p}$ to all integers $p \in \Z$ using the recurrence relation \eqref{ynn32} for $p < 0$. Note that, by symmetry of \eqref{ynn32}, we actually have
\begin{equation}\label{yn14}
a_{n,-p} = -a_{n,p}, \quad \quad \forall p \in \Z
\end{equation}
Together with the ``boundary'' conditions \eqref{ynn33} this implies
\begin{align*}
\sum_{p=-\infty}^\infty a_{n,p}U^{p} &= \sum_{p=-n}^n a_{n,p} U^p = \sum_{p=1}^n a_{n,p}(U^p-U^{-p})
\end{align*}
Hence \eqref{yn13} can be rewritten in the form 
\begin{equation}\label{yn14}
F(U,\lambda) = \sum_{n=0}^\infty \left( \sum_{p=1}^n a_{n,p} (U^p-U^{-p})\right) \lambda^n
\end{equation}
Comparing \eqref{yn14} with formula \eqref{yn12} for $i=1$, we see at once that
\[
G_1(\lambda) = \frac 1 {\{2\}} F(-q^2,-\lambda)
\]
The next lemma extends this observation to all $G_i(\lambda)$'s.

\begin{lemma}\label{ynlemma2}
For all $i \geq 1$, 
\begin{equation}\label{yn15}
G_i(\lambda) = \frac 1 {\{2\}} \sum_{k=1}^i \alpha_k^{(i)} F(-q^{2(2k-1)},-\lambda)
\end{equation}
where
\begin{equation}\label{yn15}
\alpha_k^{(i)} = (-1)^{i-k}\left[ \begin{array}{c} 2i-1\\ i-k\end{array}\right]_{q^2}
\end{equation}
\end{lemma}
\begin{proof}
Using the explicit formulas for the cyclotomic coefficients $c_{p,i-1}$ (see \eqref{eq:cyc}) and the (skew) symmetry of the $a_{n,p}$'s (see \eqref{yn14}), we write
\begin{align*}
&\phantom{=} \sum_{p=1}^n (-1)^p a_{n,p}c_{p,i-1} =\\
&= \frac 1 {\{2\}} \sum_{p=1}^n (-1)^p a_{n,p} \{2(p-(i-1))\}\cdots \{2(p-1)\}\{2p\}\{2(p+1)\}\cdots\{2(p+(i-1))\}\\
&= \frac 1 {\{2\}} \sum_{p=-n}^n a_{n,p} \{2(p-(i-1))\}\cdots \{2(p-1)\}(-q^2)^p\{2(p+1)\}\cdots \{2(p+(i-1))\}\\
&= \frac 1 {\{2\}} \left[ \sum_{p=-\infty}^\infty a_{n,p}U^p\{2(p-(i-1))\}\cdots \{2(p-1)\}\{2(p+1)\}\cdots \{2(p+(i-1))\}\right]_{U=-q^2}
\end{align*}
Since $U^p \cdot f(X^{\pm 1}) = U^pf(-q^{-2p})$ for any $f(X) \in \C[X^{\pm 1}]$, we can rewrite the last sum in the form
\[
\sum_{p=1}^n (-1)^p a_{n,p}c_{p,i-1} = \frac 1 {\{2\}} \left[ \sum_{p=-\infty}^\infty a_{n,p} U^p \cdot P^{(i)}(X)\right]_{U=-q^2}
\]
where $P^{(i)}(X) \in \C[X^{\pm 1}]$ are the Laurent polynomials defined by
\[
P_i(X) := \prod_{k=1}^{i-1}(q^{-2k}X-q^{2k}X^{-1})(q^{2k}X - q^{-2k}X^{-2k}),\quad \quad i \geq 1
\]
By formula \eqref{yn12}, we get 
\[
G_i(\lambda)=\frac 1 {\{2\}} \left[F(U,-\lambda)\cdot P^{(i)}(X)\right]_{U=-q^2}
\]
Writing the polynomials $P^{(i)}(X)$ in the form
\[
P^{(i)}(X) = b_0^{(i)} + \sum_{k=1}^{i-1}b_k^{(i)}(X^{2k}+X^{-2k})
\]
we compute
\[
F(U,-\lambda)\cdot P^{(i)}(X)= b_0^{(i)}F(U,-\lambda) + \sum_{k=1}^{i-1} b_k^{(i)}\left(F(q^{4k}U,-\lambda)+F(q^{-4k}U,-\lambda)\right)
\]
Now, substituting $U = -q^2$ and using the skew-symmetry $F(U^{-1}) = -F(U)$ of the generating series, we find
\[
\left[F(U,-\lambda)\cdot P^{(i)}(X)\right]_{U=-q^2} = \sum_{k=1}^i(b_{k-1}^{(i)}-b_k^{(i)})F(-q^{2(2k-1)},-\lambda)
\]
Whence
\[
G_i(\lambda) = \frac 1 {\{2\}} \sum_{k=1}^i \left(b_{k-1}^{(i)}-b_k^{(i)}\right)F(-q^{2(2k-1)},-\lambda)
\]

To complete the proof of the lemma, it suffices to notice that
\[
b_{k-1}^{(i)}-b_k^{(i)} = \alpha_k^{(i)}\quad \quad \forall i \geq 1,\,\, 1 \leq k \leq i-1
\]
which can be seen easily from the formula
\begin{align*}
(X-X^{-1})P^{(i)}(X) &= \prod_{k=-(i-1)}^{i-1}(q^{2k}X-q^{-2k}X^{-1}) = \sum_{k=1}^i \alpha_k^{(i)}(X^{2k-1}-X^{-2k+1}).
\end{align*}
This finishes the proof of Lemma \ref{ynlemma2}.
\end{proof}

Thus, by Lemma \ref{ynlemma2}, the generating functions $G_i(\lambda)$ are determined by the values of $F(U,\lambda)$ at $U=-q^{2(2k-1)}$ for $k\geq 1$. To compute these values we will use the functional equation 
\begin{equation}\label{yn16}
F(U,\lambda)\cdot (Y_{t_1,t_2}+Y_{t_1,t_2}^{-1} - \lambda - \lambda^{-1}) = U^{-1}-U
\end{equation}
which is equivalent to the recurrence relations \eqref{ynn32} defining the coefficients $a_{n,p}$. The equivalence of \eqref{yn16} and \eqref{ynn32} follows easily from formulas \eqref{yn9} and \eqref{yn14} and the standard generating series of  Chebyshev polynomials:
\[
\sum_{n=0}^\infty S_{n-1}(z+z^{-1})\lambda^n = -(z+z^{-1}-\lambda - \lambda^{-1})^{-1}
\]
We need one more technical lemma.

\begin{lemma}\label{ynlemma3}
For any $N \in \Z$ and any $f(U)\in \C_q[U^{\pm 1}]$, 
\begin{align}
\left[f(U)\cdot (Y_{t_1,t_2}+Y_{t_1,t_2}^{-1})\right]_{U=-q^{2N}} = 
&-(t_1q^{-2N}+t_1^{-1}q^{2N})f(-q^{2N}) \notag\\
&+ \bar t_1 \sum_{p=0}^{N-1}\{2p\}f(-q^{2(2p-N)}) \label{yn17}\\
&- \bar t_2 \sum_{p=0}^{N-1}\{2p+1\}f(-q^{2(2p-N+1)})\notag
\end{align}
\end{lemma}
\begin{proof}
Recall that the Dunkl-Cherednik operator $Y_{t_1,t_2} := Y_{t_1,t_2,1,1}$ is given explicitly by the formula (cf. \eqref{eq:T1} and Remark \ref{remark_ccother}):
\[
Y_{t_1,t_2} = t_1 Y - a(X) (Y-s)
\]
where 
\[
a(X) := \frac {q \bar t_1 X^{-1} + \bar t_2}{qX^{-1} - q^{-1}X} = \frac {\bar t_1 + \bar t_2(q^{-1}X)}{1-(q^{-1}X)^2}
\]
For any $k \in \Z$, using  \eqref{yn210} we compute
\[
U^k\cdot Y_{t_1,t_2} = t_1 U^{k-1} - a(-q^{2k})(1+U^{2k-1})U^{-k}
\]
Then, evaluating at $U=-q^{2N}$ yields
\[
\left[ U^k\cdot Y_{t_1,t_2}\right]_{U=-q^{2N}}=-t_1 q^{-2N}(-q^{2N})^k - \bar t_1 \sum_{p=0}^{N-1} q^{-2p}(-q^{2p-2N})^k+\bar t_2 \sum_{p=0}^{N-1} q^{-2p-1}(-q^{4p-2N+2})^k
\]
Hence, for any $f(U) \in \C[U^{\pm 1}]$ we have 
\begin{align}
\left[ f(U)\cdot Y_{t_1,t_2}\right]_{U=-q^{2N}} = 
&-t_1 q^{-2N}f(-q^{2N})\notag \\
&- \bar t_1 \sum_{p=0}^{N-1} q^{-qp}f(-q^{2(2p-N)})\label{yn18}\\
&+ \bar t_2 \sum_{p=0}^{N-1}q^{-2p-1}f(-q^{2(2p-N+1)})\notag
\end{align}
A similar calculation with the inverse operator 
\[
Y_{t_1,t_2}^{-1} = t_1 Y^{-1} - a(X^{-1})Y^{-1} - s) - \bar t_1 s
\] 
yields
\begin{align}
\left[ f(U)\cdot Y_{t_1,t_2}^{-1}\right]_{U=-q^{2N}} =& -t_1^{-1}q^{2N} f(-q^{2N}\notag \\
&+ \bar t_1 \sum_{p=0}^{N-1} q^{2p}f(-q^{2(2p-N)}\label{yn19}\\
&- \bar t_2 \sum_{p=0}^{N-1} q^{2p+1}f(-q^{2(2p-N+1)})\notag
\end{align}
Adding up \eqref{yn18} and \eqref{yn19} we get formula \eqref{yn17}.
\end{proof}

Now we are in a position to complete the proof of Theorem \ref{thm:main}.
\begin{proof}[Proof of Theorem \ref{thm:main}]
 Using Lemma \ref{ynlemma3}, from the functional equation \eqref{yn16} we get the system of linear equations for the values $F(-q^{2N},-\lambda)$:
\[
\gamma_N F(-q^{2N}) - \bar t_1 \sum_{p=1}^{N-1} \{2p\} F(-q^{2(2p-N)}) + \bar t_2 \sum_{p=0}^{N-1} \{2p+1\} F(-q^{2(2p-N+1)}) = -\{2N\}
\]
where $\gamma_N = -\lambda - \lambda^{-1} + q^{2N}t_1^{-1} + q^{-2N}t_1$. This system can be written in the matrix form
\begin{equation}\label{yn20}
\left( 
\begin{array}{cccc}
\gamma_1 & 0 & 0 & \cdots \\
b_{2,1} & \gamma_2 & 0 & \cdots\\
b_{3,1} & b_{3,2} & \gamma_3 & \cdots\\
\vdots & \vdots & \vdots & \ddots
\end{array}
\right)
\left(
\begin{array}{c}
F(-q^2)\\
F(-q^4)\\
F(-q^6)\\
\vdots
\end{array}
\right)
= - 
\left(
\begin{array}{c}
{\{2\}}\\
{\{4\}}\\
{\{6\}}\\
\vdots
\end{array}
\right)
\end{equation}
where 
\[
b_{p,N} := (-1)^i \left( \{p+N\}-\{p-N\}\right) \bar t_i,\quad \quad \textrm{with }\, i \equiv p-N+1\,\,(\mathrm{mod }\,\, 2)
\]
By Lemma \ref{ynlemma2}, the generating functions $G_i(\lambda)$ are given by linear combinations of solutions of this system, $F(-q^{2N},-\lambda)$, with $N = 2k-1$ for $k\geq 1$. Solving \eqref{yn20} by Cramer's Rule, we can formally express these linear combinations in terms of the matrix $B_{2i}(q,t_1,t_2;\lambda)$ described in the Introduction (see \eqref{eq:bdef}). This yields the required formulas \eqref{eq:genfun} for  $G_i(\lambda)$, finishing the proof of Theorem \ref{thm:main}.
\end{proof}

\subsection{Proof of Theorem \ref{thm:t2eq1}}\label{sec:proofthm2}
In this subsection we specialize $t_2=1$ and give an explicit formula for the coefficients $\tilde{c}_{n,i}$ in terms of classical Macdonald polynomials of type $A_1$. We begin by recalling the definition.

\begin{definition}
The \emph{Macdonald polynomials} $p_n(x;\beta | q)$, $n \geq 0$ are the symmetric orthogonal polynomials in $\C[q^{\pm 1},\beta^{\pm 1}][x+x^{-1}]$ satisfying
the 3-term recurrence relation
\[
p_{n+1} = (x+x^{-1})p_n - \frac{q^{n/2}-q^{-n/2}}{\beta^{1/2}q^{n-2}-\beta^{-1/2}q^{-n/2}} \frac {\beta q^{(n-1)/2}-\beta^{-1}q^{(1-n)/2}}{\beta^{1/2}q^{(n-1)/2}-\beta^{-1/2}q^{(1-n)/2}} p_{n-1}
\]
with 
$p_0 = 1$ and $p_1 = x+x^{-1}$.
\end{definition}


After the following renormalization\footnote{The polynomials $C_n(x;\beta | q)$ are sometimes called the $q$-ultraspherical (or Rogers) polynomials (cf. \cite[Sect. 14.10.1]{KLS10}).}
\[
C_n(x;\beta |  q) := \frac{(\beta;q)_n}{(q;q)_n}  
p_n(x;\beta | q)
\]
the Macdonald polynomials assemble into the generating series (see, e.g.  \cite{KLS10}):
\begin{equation}\label{macgenid}
\sum_{n=0}^\infty C_n(x;\beta | q) z^n = \frac{ (z \beta x; q)_\infty (z \beta x^{-1};q)_\infty}{(z x;q)_\infty (z x^{-1};q)_\infty},
\end{equation}
where 
\[
(a;q)_n := \left\{ 
\begin{array}{cl}
1 & n=0\\
\prod_{k=0}^{n-1} (1-aq^k) & 1 \leq n \leq \infty
\end{array}
\right.
\]
(For $n=\infty$ one assumes that $\lvert q \rvert < 1$.) In fact, these polynomials can be given by
\begin{equation}\label{eq:yn319p}
C_n(x;\beta | q) = \sum_{k=0}^n \frac{(\beta;q)_k (\beta;q)_{n-k}}{(q;q)_k (q;q)_{n-k}} x^{n-2k}
\end{equation}

If we specialize  $q\mapsto q^4$ and $\beta \mapsto q^{4i}$, then formulas \eqref{macgenid} and \eqref{eq:yn319p} become
\begin{equation}\label{spemacgen}
\sum_{n=0}^\infty C_n(x;q^{4i} | q^4)z^n = \frac 1 {\prod_{k=0}^{i-1} (1-q^{4k}z x)(1-q^{4k}z x^{-1})}
\end{equation}
and
\begin{equation}\label{eq:yn320p}
C_n(x;q^{4i} | q^4) = \sum_{k=0}^n \left[ \begin{array}{c} k+i-1\\i - 1 \end{array}\right]_{q^4} \left[ \begin{array}{c} n-k+i-1\\i-1\end{array}\right]_{q^4}x^{n-2k}
\end{equation}
Note that the last formula shows that $C_n(x;q^{4i} | q^4) \in \Z[q^{\pm 4}][x+x^{-1}]$ for all $n \geq 0$. 
To prove Theorem \ref{thm:t2eq1} we compare  \eqref{spemacgen} to the generating function $G_i(\lambda)$. First, we simplify the formula \eqref{eq:genfun} for $G_i(\lambda)$ given in Theorem \ref{thm:main} by explicitly computing the determinant of $B_{2i}$ in the case $t_2=1$. The result is given by the following

\begin{proposition}\label{lemma:ctclosed}
For $t_1=t$ and $t_2=1$, we have
\begin{equation}\label{eq:prop}
G_i(\lambda)  =  \frac{c_{i,i-1} \prod_{k=2}^i A_k(t)}{\prod_{k=1}^i (\lambda + \lambda^{-1} - q^{2(2k-1)}t^{-1} - q^{-2(2k-1)}t) } 
\end{equation}
where
\[
A_k(t) = \frac{q^{2k-1}t^{-1}-q^{1-2k}t}{q^{2k-1}-q^{1-2k}}
\]
\end{proposition}
\begin{proof}
We break up the proof into two steps stated as Lemmas \ref{lemma:ctgenseries} and \ref{lemma:detind} below. First, Lemma \ref{lemma:ctgenseries} shows that 
\[
\tilde c_{n,i-1} = \frac{\det{[\bar B_i]}}{\prod_{k=1}^i (\lambda + \lambda^{-1} - q^{2(2k-1)}t^{-1} - q^{-2(2k-1)}t)}
\]
where $\bar B_i$ is a certain submatrix of $B_{2k}$. Then Lemma \ref{lemma:detind} computes the determinant of $\bar B_i$ by induction, showing that
\begin{equation}\label{eq:detind}
\det{[\bar B_{i+1}]} = -\{2(2i)\}\{2(2i+1)\} A_{i+1} \det{[\bar B_i]}
\end{equation}
Together with \eqref{eq:genfun}, this gives formula \eqref{eq:prop}.
\end{proof}

\begin{lemma}\label{lemma:ctgenseries}
For all $i \geq 1$,
\[
G_i(\lambda)  = \frac{\det{[\bar B_i]}}{\prod_{k=1}^i (\lambda + \lambda^{-1} - q^{2(2k-1)}t^{-1} - q^{-2(2k-1)}t)}
\]
where $\bar B_i$ is the matrix 
\begin{equation}
\bar B_i := \left( 
\begin{array}{ccccccc}
0 & \alpha_1^{(i)} & \alpha_2^{(i)} & \cdots& & \alpha_{i-1}^{(i)} & 1 \\
\beta_1 & \gamma_1 & 0 & \cdots& & &\\
\beta_{3} & b_{3,1} & \gamma_3 & 0 & \cdots \\
\vdots & \vdots & \vdots & \ddots & 0 & \cdots\\
\beta_{2j-1} & b_{2j-1,1} & b_{2j-1,2} & \cdots  & \gamma_{2j-1} & 0 & \cdots \\
\vdots &\vdots &\vdots &\cdots & \vdots & \ddots & \\
\beta_{2i-1} & b_{2i-1,1} & b_{2i-1,2} & \cdots& \cdots & b_{2i-1,2i-2} & \gamma_{2i-1}
\end{array}
\right)
\end{equation}
\end{lemma}
\begin{proof}
Note that if $t_2=1$ and $i-j$ is even, then $b_{i,j} = 0$. This means that the second-to-last column in $B_{2i}$ has exactly one nonzero entry, which is $\gamma_{2i-2}$, located on the diagonal. Expanding the determinant along this column, we see that the same is true with the resulting $(2i-1) \times (2i-1)$ matrix. Then induction shows 
\[
 \det(B_{2i}) = \det(\bar B_i) \prod_{j=2}^i \gamma_{2j-2}
 \]
The result then follows from this identity combined with Theorem \ref{thm:main}. 
\end{proof}

\begin{lemma}\label{lemma:detind}
$\det{[\bar B_{i}]} = (-1)^i \left( \prod_{N=1}^{2i-1} \{2N\} \right) \left( \prod_{k=2}^i A_k(t)\right) $.
\end{lemma}
\begin{proof}
The proof consists of a sequence of row and column operations to show that step $\det{[\bar B_{i+1}]} = -\{2(2i)\}\{2(2i+1)\} A_{i+1} \det{[\bar B_i]}$. First, we kill all entries in the first row of $\bar B_{i+1}$ except for the last using column operations to obtain the following matrix:
\begin{equation*}
\left( 
\begin{array}{cccccc}
0 & 0 & 0 & \cdots&  & 1 \\
\beta_1 & \gamma_1 & 0 & \cdots&  \\
\vdots & \vdots  & \ddots & \cdots & \\
\beta_{2i-1} & b_{2i-1,1} & b_{2i-1,2} & \cdots  & \gamma_{2j-1}  & 0\\
\beta_{2i+1} & b_{2i+1,1} - \alpha_1^{(i+1)}\gamma_{2i+1} & b_{2i+1,2} - \alpha_2^{(i+1)}\gamma_{2i+1}& \cdots& \cdots 
& \gamma_{2i+1}
\end{array}
\right)
\end{equation*}
Then we reduce the size of this matrix by one, expanding the determinant along its first row. Next, we add $\alpha_k^{(i+1)}$ multiples of the first $i$ rows to the last row to obtain the matrix
\begin{equation*}
 (-1)^{i+3}\gamma_{2i+1} \left( 
\begin{array}{cccccc}
\beta_1 & \gamma_1 & 0 & \cdots& 0 \\
\vdots & \vdots  & \ddots & \cdots & \vdots \\
\beta_{2i-1} & b_{2i-1,1} & b_{2i-1,2} & \cdots  & \gamma_{2j-1} \\
\tilde \beta_{2i+1} & \tilde b_{2i+1,1} & \tilde b_{2i+1,2} & \cdots & \tilde b_{2i+1,i}
\end{array}
\right)
\end{equation*}
where 
\begin{align*}
\tilde \beta_{2i+1} &= \alpha_1^{(i+1)}\beta_1 + \cdots + \alpha_i^{(i+1)}\beta_{2i-1} + \beta_{2i+1}\\
\tilde b_{2i+1,k} &= b_{2i+1,k} + \alpha_{k}^{(i+1)}(\gamma_k - \gamma_{2i+1}) + \sum_{j=k+1}^{i} \alpha_{j}^{(i+1)} b_{2j-1, k} 
\end{align*}
Now, observe that by \eqref{yn20} we have $\tilde \beta_{2i+1} = 0$, so we move the last row to the top and divide it by its last entry, which is $\tilde b_{2i+1,i}$. Finally, by a straightforward computation, we check that the resulting matrix is \emph{exactly} $\bar B_i$.
\end{proof}

\begin{proof}[Proof of Theorem \ref{thm:t2eq1}]
It follows from Lemma \ref{lemma:ctclosed} that
\begin{equation}\label{eq:ctildeq}
\tilde c_{n,i-1} = \left( \frac{c_{i,i-1} \prod_{k=2}^i A_k}{\prod_{k=0}^{i-1} (\lambda + \lambda^{-1} - q^{2(2k+1)}t^{-1} - q^{-2(2k+1)}t) }  \right) [\lambda^n]
\end{equation}
where the notation $[\lambda^n]$ means the coefficient of $\lambda^n$ in the preceding expression. If we change variables $z=q^{-2(i-1)}\lambda$ and $x = q^{2i}t^{-1}$ in \eqref{spemacgen} and compare the result with \eqref{eq:ctildeq} we obtain
\begin{equation}\label{eq:yuristar}
\tilde c_{n,i-1} = c_{i,i-1} q^{-2(n-i)(i-1)} C_{n-i}(q^{2i}t^{-1}; q^{4i} | q^4) \prod_{k=2}^i A_k
\end{equation}
By specializing $t=1$ in \eqref{eq:yuristar}, we see that 
\[
C_{n-i}(q^{2i};q^{4i}|q^4) = q^{2(n-i)(i-1)}\frac{c_{n,i-1}}{c_{i,i-1}}
\]
Hence, it follows from \eqref{eq:yuristar} that 
\begin{align}
\frac{ \tilde c_{n,i-1}}{c_{n,i-1}} &= \frac{C_{n-i}(q^{2i}t^{-1};q^{4i} | q^4)}{C_{n-i}(q^{2i}; q^{4i} | q^4)} \left( \prod_{k=2}^i A_k\right) = \frac{p_{n-i}(q^{2i}t^{-1}; q^{4i} | q^4)}{p_{n-i}(q^{2i}; q^{4i} | q^4)}  \left( \prod_{k=2}^i A_k\right) \label{eq:typo}
\end{align}
This completes the proof of Theorem \ref{thm:t2eq1}.
\end{proof}

\begin{remark}\label{rmk:labelforyuri}
Using \eqref{eq:yn320p}, we can rewrite formula \eqref{eq:yuristar} in the following explicit form:
\begin{align*}
\tilde c_{n,i-1} =& (q^2t^{-1})^n \prod_{k=1}^{i-1}(q^{2k+1}+q^{-2k-1})(q^{-2k-1} t - q^{2k+1}t^{-1})\\
&\times  (q^8;q^8)_{i-1} \sum_{k=0}^{n-i} \left[ \begin{array}{c} k+i-1\\i-1\end{array}\right]_{q^4} \left[ \begin{array}{c} n-k-1\\i-1\end{array}\right]_{q^4} (q^{-2i} t)^{i+2k}
\end{align*}
which makes the integrality of $\tilde c_{n,i-1}$ (Corollary \ref{cor:integral}) obvious.
\end{remark}

\subsection{Proof of Theorem \ref{thm:uni}}\label{sec:uni}
Theorem \ref{thm:uni} follows easily by comparing our results (specifically Lemma \ref{ynlemma1}) with Habiro's results proved in  \cite{Hab08}. For the reader's convenience (and to avoid confusion with notation), we will state Habiro's main theorem on universal $\sl_2$-invariants below. First, we recall from the Introduction that $\U_\hhbar = \U_\hhbar(\sl_2)$ stands for the quantized universal enveloping algebra of the Lie algebra $\sl_2$: this is an $\hhbar$-adically complete $\Q[[\hhbar]]$-algebra (topologically) generated by elements $E, F, H$ satisfying the relations
\begin{equation}\label{ynnn1}
[H,E] = 2E,\quad [H,F]=-2F,\quad [E,F] = \frac{K-K^{-1}}{v-v^{-1}}
\end{equation}
where $v := e^{\hhbar/2}$ and $K := v^H = e^{\hhbar H/2}$. This algebra carries a natural (complete) ribbon Hopf algebra structure with universal $R$-matrix given by 
\begin{equation}\label{ynnn2}
R = v^{H\otimes H/2} \sum_{n \geq 0} v^{n(n-1)/2} \frac{(v-v^{-1})^n}{[n]_{v}!} E^n \otimes F^n
\end{equation}
(where we have used the notation $[n]_q!  := \prod_{k=1}^n [k]_q$).
Using the  $R$-matrix \eqref{ynnn2}, for any (ordered, oriented, framed) link $L$ in $S^3$, R.\ Lawrence \cite{Law88, Law90} constructed a link invariant $J^L$, called the \emph{universal $\sl_2$-invariant}\footnote{Lawrence's universal invariants can be defined for more general Lie algebras than $\sl_2$ and for more general link-type diagrams (bottom tangles), see \cite{Hab06}.} of $L$. If $L$ has $\ell$ components, the Lawrence invariant $J^L$ takes its values in $\U_\hhbar^{\hat \otimes \ell}$, the $\hhbar$-adically completed tensor product of $l$ copies of $\U_\hhbar$. In the case of knots (i.e. a link $K$ with a single component), the Lawrence invariant $J^K$ is contained in the center $\mathcal Z(\U_\hhbar)$, which is a complete commutative subalgebra of $\U_\hhbar$ (topologically) freely generated by the Casimir element $C = (v-v^{-1})^2 FE + (vK + v^{-1}K^{-1} - v -  v^{-1})$.

Habiro found a general formula for $J^K$ expressing it in terms of polynomials $H_k^K(v) \in \Z[v^2,v^{-2}]$: 
\begin{theorem}[{\cite[Theorem 4.5]{Hab08}}] \label{thm:ydots}
For any (string, $0$-framed) knot $K$, the Lawrence universal $\sl_2$-invariant is given by 
\begin{equation}\label{ynnn3}
J^K = \sum_{k = 0}^\infty H_k^K(v)\sigma_k
\end{equation}
where
\[
\sigma_k = \prod_{i=1}^k (C^2-(v^i+v^{-i})^2) \in \mathcal Z(\U_\hhbar),\quad \quad k \geq 0
\]
\end{theorem}
Now, Habiro's Theorem \ref{thm:hab} stated in the Introduction follows from Theorem \ref{thm:ydots} by evaluating the elements $\sigma_k$ on finite dimensional irreducible representations of $\U_\hhbar$ using quantum traces. It is well known that such representations $V_n$  are classified by the non-negative integers -- the dimension (i.e. the rank of $V_n$ as a free module over $\Q[[\hhbar]]$). Recall, for a finite dimensional representation $\rho_V: \U_\hhbar \to \mathrm{End}_{\Q[[\hhbar]]}(V)$ and an element $u \in \U_\hhbar$, the quantum trace $\tr_q^V(u)$ is defined by 
\begin{equation}\label{ynnn4}
\tr_q(u,v) := \Tr_V[\rho_V(Ku)]
\end{equation}
where $\Tr_V$ is the usual (matrix) trace on $V$. For central elements $z \in \mathcal Z(\U_\hhbar)$ one can compute \eqref{ynnn4} using the Harish-Chandra homomorphism
\[
\mathcal Z(\U_\hhbar) \hookrightarrow \U_\hhbar \to \Q[[h]][H]
\]
defined (on the PBW basis of $\U_\hhbar$) by 
\[
\varphi(F^i H^j E^k) = \delta_{i,0}\delta_{k,0} H^j
\]
Specifically, for any $n \geq 1$, we have 
\begin{equation}\label{ynnn5}
\tr_q(z,V_n) = \dim_v(V_n) \ev_n(\varphi(z)) = [n]_v \mathrm{ev}_n (\varphi(z))
\end{equation}
where $\mathrm{ev}_n: \Q[[\hhbar]][H] \to \Q[[h]]$ is the evaluation map  $f(H) \mapsto f(n)$.

Using formula \eqref{ynnn5}, it is straightforward to show that
\[
\tr_q(\sigma_k; V_n) = \frac{1} {v-v^{-1}}\prod_{p=n-k}^{n+k}(v^p-v^{-p}),\quad \quad \forall k \geq 1
\]
Thus, setting $v = q^2$, we obtain\footnote{We warn the reader that our $q$ differs from the $q$ in \cite{Hab08}: in fact, the $q$ in \cite{Hab08} equals $v^2$, which is our $q^4$.} 
\begin{equation}\label{ynnn6}
\tr_q(\sigma_k,V_n) = c_{n,k}(q)
\end{equation}
where $c_{n,k}(q)$ are precisely the cyclotomic coefficients \eqref{eq:cyc}.  If follows from Theorem \ref{thm:ydots} and formula \eqref{ynnn6} that
\begin{equation}\label{ynnn7}
\hat J^K(V_n) = \tr_q(J^K, V_n) = \sum_{k\geq 0} c_{n,k}(q) H_k^K(q) = J_n^K(q)
\end{equation}
Now, the proof of Theorem \ref{thm:uni} reduces to the one line calculation
\begin{align*}
\hat J^K(\tilde V_n) &= \sum_{p=1}^n (-1)^{n+p} a_{n,p} \, \hat J^K(V_n)\\
&\stackrel{\eqref{ynnn7}}{=} \sum_{p=1}^n (-1)^{n+p} a_{n,p}\, J_p^K(q)\\
&= J_n^K(q,t_1,t_2)
\end{align*}
where the last equality is formula \eqref{yn6} of Lemma \ref{ynlemma1}.

\begin{remark}
One might wonder why an invariant defined by a DAHA action on a skein module could be expressed in terms of the representation ring of $\U_q(\sl_2)$. A brief explanation for this is as follows: consider the \emph{Temperly-Lieb category}, which is a monoidal category whose objects are the natural numbers and whose morphisms from $m$ to $n$ are the $(m,n)$-tangles\footnote{An $(m,n)$-tangle is a properly embedded 1-manifold in $[0,1]\times [0,1]$ with $m$ endpoints on $\{0\}\times [0,1]$ and $n$ endpoints on $\{1\} \times [0,1]$.} in $[0,1]\times [0,1]$ regarded modulo the Kauffman bracket skein relations. The monoidal structure comes from addition on objects, and on morphisms is defined using juxtaposition of disks. It is a classical fact (see \cite{Kup96, Tin17, CKM14} and references therein) that (the Karoubi envelope of) the Temperly-Lieb category is equivalent to  the category $\Rep( \U_q(\sl_2)) $ of finite dimensional representations of $\U_q(\sl_2)$. This implies that there is a natural map $HH_0(\Rep(\U_q(\sl_2))) \to K_q(S^1\times D^2)$ from the Hochschild homology of $\Rep(\U_q(\sl_2))$ to the skein module of (closed) loops in the annulus, which is actually an isomorphism. 
On the other hand, for any semisimple category $C$, there is a canonical (Chern character) map $\mathrm{ch}:K_0(C) \to HH_0(C)$ which becomes an isomorphism upon linearization of $K_0(C)$. This
means in our case that we can naturally identify the representation ring $R_q := K_0(\Rep(\U_q(\sl_2)))$ with the skein algebra $K_q(S^1\times D^2)$. 
As a result, for a knot $K$ we get a commutative diagram 
\[
\begin{tikzcd}
R_q \arrow{r}{\mathrm{ch}} \arrow[swap]{dr}{\hat J^K} & K_q(S^1\times D^2) \arrow{d}{\langle -,\varnothing\rangle} \\
& \Q[[\hhbar]]
\end{tikzcd}
\]
which leads to formula \eqref{ynnn7}.
\end{remark}


Finally, we say a few words about our motivation for this paper.
One of the principal problems in quantum topology is to relate link invariants constructed using representation theory (in particular, the theory of quantum groups and related quantum algebras) to invariants of 3-manifolds coming from geometry. One outstanding conjecture in this direction is the so-called Volume Conjecture, which can be stated as follows: 

\begin{conjecture*}[\cite{Kas97, MM01}]
For any hyperbolic knot $K$ in $S^3$, 
\begin{equation}\label{eq:volconj}
 \lim_{n \to \infty} \frac{1}{n} \log \left\lvert \frac { J^K_n(e^{ \pi i / 2n}) } { J^U_n(e^{\pi i / 2n})}\right \rvert  = \frac 1 {2\pi} \mathrm{Vol}(S^3 \setminus K) 
\end{equation}
where $\mathrm{Vol}(S^3\setminus K)$ is the (hyperbolic) volume of the knot complement $S^3 \setminus K$ (see, e.g. \cite{GL11}).
\end{conjecture*}

This conjecture has been confirmed in a number of examples, but the general case is completely open (see \cite{Mur11} for a survey). 
The existence of the generalized Jones polynomials $J_n^K(q,t_1,t_2)$ naturally lead us to the following
\begin{question*}
	Does a limit of the form \eqref{eq:volconj} exist for the polynomials $J_n^K(q,t_1,t_2)$ when $t_1\not= 1$ and/or $t_2 \not= 1$? If so, what is its geometric meaning?
\end{question*}

The explicit formulas for $J_n^K(q,t_1,t_2)$ constructed in this paper open up the way for studying the above question: we plan to address it in our future work.

In the end, we would like to mention that, for algebraic knots, there
are other interesting generalizations of Jones polynomials
based on representation theory of double affine Hecke algebras (see,
e.g., 
\cite{Che13}, \cite{CD16}, \cite{GN15}).
The precise connection between
these generalized DAHA Jones polynomials and the ones proposed in \cite{BS16}
is still unclear. We hope that the results of this paper
will help to clarify this question.

\bibliography{draft_bibtex}{}
\bibliographystyle{amsalpha}

\end{document}